\documentclass[12pt]{article}
\usepackage[english]{babel}
\usepackage{hyperref}
\usepackage{indentfirst}
\usepackage{amsfonts,amssymb,amsthm,amsmath}
 \theoremstyle{theorem}
\newtheorem*{Th}{Theorem}
\theoremstyle{lemma}
\newtheorem{Lm}{Lemma}
\theoremstyle{plain}
\newtheorem{coll}{Corollary}

\newcommand{\probel}{\mbox{ }}
\newcommand{\ir}{[t_0,\vartheta_0]\times \mathbb{R}^n}
\newcommand{\iro}{(t_0,\vartheta_0)\times \mathbb{R}^n}

\newcommand{\SDup}[2]{D^+_{\rm D}{#1}({#2})}
\newcommand{\SDdn}[2]{D^-_{\rm D}{#1}({#2})}

\newcommand{\pis}[2]{[t_0,\vartheta_0]\times \Pi(e_{#1},{#2})\times S^{(n-1)}}

\newcommand{\sgn}{{\rm sgn}}
\newcounter{tmp}
\newcounter{tmp1}
\newcounter{tmp2}

\newcounter{tmp4}

\begin{document}
\title{On a Structure of the Set of Differential Games Values\footnote{Work is supported by RFBR (grants No 06-01-00414, 07-01-96088).}}%

\author{Yurii Averboukh} 
\date{\fontsize{10}{12}\selectfont Institute of Mathematics and Mechanics UrB RAS\\
S. Kovalevskaya street 16\\
620219, GSP-384, Ekaterinburg
Russia\\
ayv@imm.uran.ru, averboukh@gmail.com}
\maketitle
\begin{abstract}
In this paper the set of value functions of all-possible zero-sum differential games with terminal payoff is characterized. The necessary and sufficient condition  for a given function to be a value of some differential game with terminal payoff is obtained.
\end{abstract}

\section{Introduction}
The paper is devoted to the theory of two-controlled, zero sum differential games. Within the framework of this theory the control processes  under uncertainty are studied. N.N. Krasovskii and A.I. Subbotin introduced the feedback formalization of differential games \cite{NN_PDG_en}. This formalization allows them to prove the existence of value function.

In this paper we characterize the set of value functions of all-possible zero-sum differential games with terminal payoff. The value function is minimax (or viscosity) solution of corresponding Isaacs-Bellman equation (Hamilton-Jacobi equation) \cite{Subb_book}.

One can consider a differential game within usual constraints as a complex of two control spaces, game dynamic and terminal payoff function.  The time interval and state space of game are assumed to be fixed.
In this paper the following problem is considered: let the locally lipschitzian function $\varphi(t,x)$ be given, do there exist control spaces, dynamic function and terminal payoff function such that the function $\varphi(t,x)$ is the value of corresponding differential game?

\section{Preliminaries}
In this section we recall the main notions of the theory of zero-sum differential games. We follow the formalization of N.N.~Krasovskii and A.I.~Subbotin.

Usually in the theory of differential games the following problem is considered \cite{NN_PDG_en}. Let the controlled system \begin{equation}\label{system}
    \dot{x}=f(t,x,u,v), \ \ t\in [t_0,\vartheta_0], \ \ x\in\mathbb{R}^n, \ \ u\in P,\ \ v\in Q
\end{equation} and payoff functional $\sigma(x(\vartheta_0))$ be given. Here $u$ and $v$ are controls of the player $U$ and the player $V$ respectively. The player $U$ tries to minimize the payoff and the player $V$ wishes to maximize the payoff.  The purpose is to find the value of corresponding game. The value is a function $\varphi$ from $\ir$ to $\mathbb{R}$.

Suppose that $P$ and $Q$ are finite-dimensional compacts.
The function $f$ satisfies the following assumption:
\begin{list}{F\arabic{tmp}.}{\usecounter{tmp}}
    \item $f$ is continuous;
    \item $f$ is locally lipschitzian with respect to the phase variable;
    \item there exists constant $\Lambda_f$ such that for every $t\in [t_0,\vartheta_0]$, $x\in\mathbb{R}^n$, $u\in P$,
    $v\in Q$ the following inequality holds:
     $$\|f(t,x,u,v)\|\leq \Lambda_f(1+\|x\|). $$
\end{list}
Often the Isaacs condition is put: for any $t\in [t_0,\vartheta_0]$, $x\in\mathbb{R}^n$,
$s\in \mathbb{R}^n$ the equality $$\min_{u\in P}\max_{v\in Q}\langle
s,f(t,x,u,v)\rangle=\max_{v\in Q}\min_{u\in P}\langle s,f(t,x,u,v)\rangle $$ is valid.

The function $\sigma:\mathbb{R}^n\rightarrow \mathbb{R}$ satisfies the following assumption (see \cite{Subb_book}, \cite{Bardi}):
\begin{list}{$\Sigma$\arabic{tmp1}.}{\usecounter{tmp1}}
\item $\sigma$ is locally lipshitzian;
\item there exists $\Lambda_\sigma$ such that $$|\sigma(x)|\leq \Lambda_\sigma(1+\|x\|). $$
\end{list}

Assumption $\Sigma$1 grantees the locally lipschitzness of value function. Assumption $\Sigma$1 is often replaced by the condition of continuity of $\sigma$. Assumption $\Sigma$2 was used by A.I.~Subbotin in his theory of minimax solution. It is not traditional for other approaches.

We consider three types of control design \cite{NN_PDG_en}.
\begin{enumerate}
  \item Player $U$ chooses the control in the class of counter-stratagies, and the player $V$ chooses the control in the class of feedback strategies.
  \item Player $U$ chooses the control in the class of  feedback strategies, and the player $V$ chooses the control in the class of counter-stratagies.
  \item Isaacs condition is valid and players $U$ and $V$ choose the controls in the classes of feedback strategies.
\end{enumerate}

N.N.~Krasovskii and A.I.~Subbotin proved that value functions are well-defined in these three cases. Let us denote the value function in the first case by $Val^{f}(\cdot,\cdot,P,Q,f,\sigma)$, in the second case by $Val^{s}(\cdot,\cdot,P,Q,f,\sigma)$, in the third case by $Val(\cdot,\cdot,P,Q,f,\sigma)$. It is well-known that value functions are locally lipshitzian under assumption F1--F3, $\Sigma$1, $\Sigma$2 \cite{Bardi}.

A.I. Subbotin proved that the value of differential game satisfies the boundary condition
\begin{equation}\label{boundar_cond}
    \varphi(\vartheta_0,x)=\sigma(x)
\end{equation} and the equation
\begin{equation}\label{HJ}
    \frac{\partial\varphi(t,x)}{\partial t}+H(t,x,\nabla\varphi(t,x))=0
\end{equation} in generalized sense.
Here $\nabla\varphi(t,x)$ means the vector of partial derivatives of $\varphi$ with respect to space variables.

$H$ is called Hamiltonian of differential game. It is defined in the following way.
\begin{itemize}
\item In the first case $H$ is given by
$$H(t,x,s)=H^{(-)}(t,x,s)\triangleq\max_{v\in Q}\min_{u\in P}\langle s,f(t,x,u,v)\rangle. $$
\item In the second case $H$ is given by $$H(t,x,s)=H^{(+)}(t,x,s)\triangleq \min_{u\in P}\max_{v\in Q}\langle s,f(t,x,u,v)\rangle. $$
\item If Isaacs condition is valid, then $H^{(-)}=H^{(+)}$. Therefore, $H=H^{(-)}=H^{(+)}$ in this case.
\end{itemize}

A.I.Subbotin introduced several definitions of generalized (minimax) solution of Hamilton-Jacobi equation \cite{Subb_book}. He proved that they are equivalent. Also A.I.~Subbotin proved that notion of minimax solution coincides with the notion of viscosity solution (see \cite{Subb_book} and \cite{Bardi}). We use one of equivalent definitions of minimax solution. Function
$\varphi(t,x)$ is called minimax solution of Hamilton-Jacobi equation (\ref{HJ}),
if for every $(t,x)\in (t_0,\vartheta_0)\times \mathbb{R}^n$ the following conditions is fulfilled:
\begin{equation}\label{U4}
    a+H(t,x,s)\leq 0\probel\probel \forall (a,s)\in \SDdn{\varphi}{t,x};
\end{equation}
\begin{equation}\label{L4}
    a+H(t,x,s)\geq 0\probel\probel \forall (a,s)\in \SDup{\varphi}{t,x};
\end{equation}

Here we use the notions of nonsmooth analysis \cite{DemRub}. Sets $\SDdn{\varphi}{t,x}$ and $\SDup{\varphi}{t,x}$ are called Dini subdifferential and Dini superdifferential respectively. They are defined by following rules. \begin{multline*}
\SDdn{\varphi}{t,x}\triangleq\Bigl\{(a,s)\in\mathbb{R}\times\mathbb{R}^n:\\ a\tau+\langle s, g\rangle \leq \liminf_{\alpha\rightarrow 0} \frac{\varphi(t+\alpha\tau , x+\alpha g)-\varphi(t,x)}{\alpha}\ \ \forall(\tau,g)\in\mathbb{R}\times\mathbb{R}^n \Bigr \},
\end{multline*}
\begin{multline*}
\SDup{\varphi}{t,x}\triangleq\Bigl\{(a,s)\in\mathbb{R}\times\mathbb{R}^n:\\ a\tau+\langle s, g\rangle \geq \limsup_{\alpha\rightarrow 0} \frac{\varphi(t+\alpha\tau , x+\alpha g)-\varphi(t,x)}{\alpha} \ \ \forall(\tau,g)\in\mathbb{R}\times\mathbb{R}^n \Bigr\}.
\end{multline*}

The function $\varphi$ is locally lipshitzian, since $\sigma$ is locally lipshitzian \cite{Bardi}. There exists a differentiability set of $\varphi$, denote it by $J$. We have $J\subset\iro$. By the Rademacher's theorem \cite{Evans_measure} measure $(\ir)\setminus J$ is 0, therefore the closure of $J$ is equal to $\ir$. For $(t,x)\in
J$ full derivative of $\varphi$ is $(\partial \varphi(t,x)/\partial t, \nabla\varphi(t,x))$. If $\SDup{\varphi}{t,x}$ and $\SDdn{\varphi}{t,x}$ are nonempty simultaneously, then $\varphi$ is differentiable at $(t,x)$, and $\SDup{\varphi}{t,x}=\SDdn{\varphi}{t,x}=\{(\partial\varphi(t,x)/t,\nabla\varphi(t,x))\}$ \cite{DemRub}.

If $\varphi$ is differentiable at position  $(t,x)$, then equality (\ref{HJ}) is valid at the position $(t,x)$ in the ordinary sense.

Let $(t,x)\in \ir$. Consider the set
\begin{equation*}\label{A_def}
\mathcal{A}\varphi(t,x)\triangleq\left\{(a,s):\exists \{(t_i,x_i)\}_{i=1}^\infty \subset J: \ \ a=\lim_{i\rightarrow\infty}\frac{\partial\varphi(t_i,x_i)}{\partial t},\ \ s=\lim_{i\rightarrow\infty}\nabla\varphi(t_i,x_i) \right\}.
\end{equation*}

Since $\varphi$ is locally lipshitzian, the set $\mathcal{A}\varphi(t,x)$ is equal to Clarke subdifferential at the position $(t,x)$ \cite{DemRub}.
Therefore, we have \cite{DemRub}
\begin{equation}\label{sub_dif_incl}
    \SDdn{\varphi}{t,x},\SDup{\varphi}{t,x}\subset\mathcal{A}\varphi(t,x).
\end{equation}

Let us describe the properties of Hamiltonian.

First, let us introduce a class of real-valued function. This class will be used extensively throughout this  paper. Denote by $\Omega$  the set of all even semiadditive  functions $\omega:\mathbb{R}\rightarrow [0,+\infty)$ such that $\omega(\delta)\rightarrow 0$, $\delta\rightarrow 0$.

If $H=H^{(-)}$ or $H=H^{(+)}$ then the following conditions are valid with $\Upsilon=\Lambda_f$ (see \cite{Subb_book}):
\begin{list}{H\arabic{tmp2}.}{\usecounter{tmp2}}
\item (sublinear growth condition) for all  $(t,x,s)\in \mathbb{R}^n$
$$|H(t,x,s)|\leq \Upsilon\|s\|(1+\|x\|); $$
\item
for every bounded region $A\subset\mathbb{R}^n$ there exist function $\omega_A\in \Omega$ and constant $L_A$ such that for all
$(t',x',s'),(t'',x'',s'')\in [t_0,\vartheta_0]\times A\times \mathbb{R}^n$, $\|s'\|,\|s''\|\leq R$ the following inequality holds:
\begin{multline*}
\|
H(t',x',s')-H(t'',x'',s'')\|\leq\\ \leq
\omega_A(t'-t'')+L_AR\|x'-x''\|+\Upsilon(1+\inf\{\|x'\|,\|x''\|\})\|s_1-s_2\|;
\end{multline*}
\item $H$ is positively homogeneous with respect to the third variable: $$H(t,x,\alpha s)=\alpha H(t,x,s)\probel \forall
(t,x)\in \ir,\probel\forall s\in \mathbb{R}^n\probel \forall \alpha\in [0,\infty).
 $$
\end{list}

\section{Main Result}\label{sect_main}
In this section we study the class of functions which may be a values of differential game. The main result is formulated below.

Denote by $\mathrm{COMP}$ the set of all finite-dimensional compacts. Let $P,Q\in\mathrm{COMP}$, denote by $\mathrm{DYN}(P,Q)$ the set of all functions
$f:[t_0,\vartheta_0]\times\mathbb{R}^n\times P\times Q\rightarrow \mathbb{R}^n$
satisfying the conditions F1--F3. Denote by  $\mathrm{DYNI}(P,Q)$ the set of all functions
$f:[t_0,\vartheta_0]\times\mathbb{R}^n\times P\times Q\rightarrow \mathbb{R}^n$
satisfying Isaacs condition and conditions F1--F3.
The set of functions $\sigma:\mathbb{R}^n\rightarrow\mathbb{R}$ satisfying condition $\Sigma$1
and $\Sigma$2 is denoted by $\mathrm{TP}$.

The set of values of differential games may be described in the following way.

\begin{itemize}\item Set of values of differential games considered in the class counter-strategy/strategy is
\begin{multline*}
\mathrm{VALF}=\{\varphi:[t_0,\vartheta_0]\times
\mathbb{R}^n\rightarrow\mathbb{R}:\\ \exists P,Q\in\mathrm{COMP}\exists f\in
\mathrm{DYN}(P,Q)\exists \sigma\in \mathrm{TP}:\probel
\varphi=Val^f(\cdot,\cdot,P,Q,f,\sigma)\}.
\end{multline*}
\item Set of values of differential games considered in the class strategy/counter-strategy is
\begin{multline*}
\mathrm{VALS}=\{\varphi:[t_0,\vartheta_0]\times
\mathbb{R}^n\rightarrow\mathbb{R}:\\ \exists P,Q\in\mathrm{COMP}\exists f\in
\mathrm{DYN}(P,Q)\exists \sigma\in \mathrm{TP}:\probel
\varphi=Val^s(\cdot,\cdot,P,Q,f,\sigma)\}.
\end{multline*}
\item Set of values of differential games considered in the class of feedback strategies is
\begin{multline*}
\mathrm{VALI}=\{\varphi:[t_0,\vartheta_0]\times
\mathbb{R}^n\rightarrow\mathbb{R}:\\ \exists P,Q\in\mathrm{COMP}\exists f\in
\mathrm{DYNI}(P,Q)\exists \sigma\in \mathrm{TP}:\probel
\varphi=Val(\cdot,\cdot,P,Q,f,\sigma)\}.
\end{multline*}
\end{itemize}

Denote by  $\mathrm{Lip}_B$ the set of all locally lipschitzian functions $\varphi:\ir\rightarrow \mathbb{R}$ such that $\varphi(\vartheta_0,\cdot)$ satisfies sublinear growth condition. The sets $\mathrm{VALF}$, $\mathrm{VALS}$,
$\mathrm{VALI}$ are subset of the set $\mathrm{Lip}_B$. Also, $\mathrm{VALI}\subset \mathrm{VALF}$ and $\mathrm{VALI}\subset \mathrm{VALS}$.

Let $\varphi\in \mathrm{Lip}_B$. Denote the differentiability set of $\varphi$ by $J$.
  For
$(t,x)\in J$ set $$E_1(t,x)\triangleq \{\nabla \varphi(t,x)\};$$
\begin{equation}\label{h_def_nabla} h(t,x,\nabla\varphi(t,x))\triangleq
-\frac{\partial\varphi(t,x)}{\partial t}.
\end{equation}

Put the following condition.

\begin{list}{(E\arabic{tmp4})}{\usecounter{tmp4}}\item For any position $(t_*,x_*)\notin J$, and any sequences  $\{(t'_i,x'_i)\}_{i=1}^\infty,\ \ \{(t''_i,x''_i)\}_{i=1}^\infty\ \ \subset J$ such that $(t'_i,x'_i)\rightarrow (t_*,x_*)$,
$i\rightarrow\infty$, $(t''_i,x''_i)\rightarrow (t_*,x_*)$, $i\rightarrow\infty$, the following implication holds:
\begin{multline*}
(\lim_{i\rightarrow\infty}\nabla\varphi(t'_i,x'_i)=\lim_{i\rightarrow\infty}\nabla\varphi(t''_i,x''_i))\Rightarrow\\
(\lim_{i\rightarrow\infty}h(t'_i,x'_i,\nabla\varphi(t'_i,x'_i))=\lim_{i\rightarrow\infty}h(t''_i,x''_i,\nabla\varphi(t''_i,x''_i))).
\end{multline*}\end{list}

Let $(t,x)\in \ir\setminus J$, denote $$E_1(t,x)=\{s\in\mathbb{R}^n:\exists
\{(t_i,x_i)\}\subset J: \lim_{i\rightarrow \infty}(t_i,x_i)=(t,x)\probel \&\probel
\lim_{i\rightarrow\infty}\nabla\varphi(t_i,x_i)=s\}.
$$ Since $\varphi$ is locally lipschitzian, the set $E_1(t,x)$
is nonempty and bounded for every $(t,x)\in \ir\setminus J$.

If $(t,x)\in \ir\setminus J$ and $s\in E_1(t,x)$, then assumption (E1) yield that the following value is well defined:
\begin{multline}\label{h_def_lim}
h(t,x,s)\triangleq\lim_{i\rightarrow\infty}h(t_i,x_i,\nabla\varphi(t_i,x_i))\\
\forall \{(t_i,x_i)\}_{i=1}^\infty\subset J:\lim_{i\rightarrow\infty}(t_i,x_i)=(t,x)\probel\&\probel
s=\lim_{i\rightarrow\infty}\nabla\varphi(t_i,x_i).
\end{multline}

Condition (E1) is the condition of extendability $h$ from the set $\mathbb{E}_0\triangleq\{(t,x,\nabla\varphi(t,x)):(t,x)\in J\}$ to $\overline{\mathbb{E}_0}\cap\{(t,x,s):(t,x)\in\ir\setminus J\}$. Thus, function $h$ is defined on the basis of Clarke subdifferential of $\varphi$ at $(t,x)\in\iro\setminus J$.
Indeed, Clarke subdifferential of $\varphi$ at $(t,x)\notin J$ is equal to
\begin{equation}\label{cl_repr_s1}
\mathcal{A}{\varphi}(t,x)={\rm co}\{(-h(t,x,s),s):s\in E_1(t,x)\}.
\end{equation}

Recall that for any $(t,x)\in \iro$
\begin{equation}\label{incl_dini_clark}
    \SDdn{\varphi}{t,x},\SDup{\varphi}{t,x}\subset \mathcal{A}{\varphi}(t,x).
\end{equation}

Denote $$CJ^-\triangleq\{(t,x)\in \iro\setminus
J:\SDdn{\varphi}{(t,x)}\neq\varnothing\};$$
$$CJ^+\triangleq\{(t,x)\in \iro\setminus J:\SDup{\varphi}{(t,x)}\neq\varnothing\}.$$ Notice that $CJ^-\cap
CJ^+=\varnothing$.

Define 
a set $E_2(t,x)$ for $(t,x)\in CJ^-$ by the rule:
$$E_2(t,x)\triangleq \{s\in\mathbb{R}^n:\exists a\in\mathbb{R}: (a,s)\in
\SDdn{\varphi}{(t,x)}\}\setminus E_1(t,x).$$ If $(t,x)\in CJ^+$ set
$$E_2(t,x)\triangleq \{s\in\mathbb{R}^n:\exists a\in\mathbb{R}: (a,s)\in
\SDup{\varphi}{(t,x)}\}\setminus E_1(t,x).$$ If $(t,x)\in(\ir)\setminus (CJ^-\cup
CJ^+)$ set $$E_2(t,x)\triangleq\varnothing. $$

The set $E_2(t,x)$ is complement of $E_1(t,x)$ with respect to projection of Dini subdifferential (or superdifferential) at $(t,x)$.

For $(t,x)\in \ir$ define $$E(t,x)\triangleq E_1(t,x)\cup E_2(t,x).$$
$E(t,x)\neq \varnothing$ for any $(t,x)\in \ir$.

Let us introduce the following notations. If $i=1,2$, then $$\mathbb{E}_i\triangleq \{(t,x,s):(t,x)\in \ir, \ \ s\in E_1(t,x)\}.$$ Denote
$$\mathbb{E}\triangleq \{(t,x,s):(t,x)\in \ir, \ \ s\in E(t,x)\};$$
$$\mathbb{E}^\natural\triangleq \{(t,x,s):(t,x)\in \ir, \ \ s\in E^\natural(t,x)\}. $$ Note that $\mathbb{E}^\natural\subset\ir\times S^{(n-1)}$.
Here $S^{(n-1)}$ means $(n-1)$-dimensional sphere $$S^{(n-1)}\triangleq\{s\in\mathbb{R}^n:\|s\|=1\}.$$
Also, $\mathbb{E}=\mathbb{E}_1\cup \mathbb{E}_2$.

Note, that the function $h$ is defined on $\mathbb{E}_1$.  The truth of inclusion $\varphi\in\mathrm{VALF}$ depends on the existence of this extension of $h$ to $\mathbb{E}$.

\begin{Th}\label{Th_main}Function $\varphi\in \mathrm{Lip}_B$ belongs to the set
$\mathrm{VALF}$ if and only if the condition {\rm (E1)} holds and  the function $h$ defined on $\mathbb{E}_1$
by formulas (\ref{h_def_nabla}) and (\ref{h_def_lim}) is extendable to the set $\mathbb{E}$
such that conditions {\rm (E2)--(E4)} are valid. (Conditions {\rm (E2)--(E4)} are defined below.)
\end{Th}

\begin{list}{(E\arabic{tmp4})}{\usecounter{tmp4}}\setcounter{tmp4}{1}
\item\label{E_solution}
\begin{itemize}\item If $(t,x)\in CJ^-$ then for any $s_1,\ldots s_{n+2}\in E_1(t,s)$
$\lambda_1,\ldots, \lambda_{n+2}\in[0,1]$ such that $\sum \lambda_k=1$, $(-\sum\lambda_k
h(t,x,s_k),\sum \lambda_k s_k)\in D^-\varphi(t,x)$ the following inequality holds: $$h \left(t,x,\sum_{k=1}^{n+2}
\lambda_k s_k\right)\leq \sum_{k=1}^{n+2}\lambda_k h(t,x,s_k).
$$
\item If $(t,x)\in CJ^+$ then for any $s_1,\ldots s_{n+2}\in E_1(t,s)$
$\lambda_1,\ldots, \lambda_{n+2}\in[0,1]$ such that $\sum \lambda_k=1$, $(-\sum\lambda_k
h(t,x,s_k),\sum \lambda_k s_k)\in D^+\varphi(t,x)$ the following inequality holds: $$h\left(t,x,\sum_{k=1}^{n+2}
\lambda_k s_k\right)\geq \sum_{k=1}^{n+2}\lambda_k h(t,x,s_k).
$$\end{itemize}
\end{list}

The condition (E2) is an analog of minimax inequalities (\ref{U4}), (\ref{L4}).

\begin{list}{(E\arabic{tmp4})}{\usecounter{tmp4}}\setcounter{tmp4}{2}
\item\label{E_gomogenious}
For all $(t,x)\in \ir$:
\begin{itemize}
    \item if $0\in E(t,x)$, then $h(t,x,0)=0$;
    \item if $s_1\in E(t,x)$ and $s_2\in E(t,x)$ are codirectional (i.e. $\langle s_1,s_2\rangle=\|s_1\|\cdot \|s_2\|$), then  $$\|s_2\|h(t,x,s_1)=\|s_1\|h(t,x,s_2). $$
  \end{itemize}
\end{list}

This condition means that function $h$ is positively homogeneous with respect to $s$.

Let us introduce the function $h^\natural(t,x,s):\mathbb{E}^\natural\rightarrow\mathbb{R}$. Put $\forall (t,x)\in
\ir\probel\forall s\in E(t,x)\setminus\{0\}$
\begin{equation}\label{h_1_def}
h^\natural(t,x,\|s\|^{-1}s)\triangleq \|s\|^{-1}h(t,x,s).
\end{equation} Under condition (E3) the function $h^\natural$ is well defined.

\begin{list}{(E\arabic{tmp4})}{\usecounter{tmp4}}\setcounter{tmp4}{3}
\item\label{E_h_bicar}
\begin{itemize}
\item Function $h^\natural$ satisfies the sublinear growth condition: there exists $\Gamma>0$ such that for any $(t,x,s)\in \mathbb{E}^\natural$ the following inequality is fulfilled $$h^\natural(t,x,s)\leq \Gamma(1+\|x\|).$$
\item For every bounded region $A\subset \mathbb{R}^n$ there exist $L_A>0$ and function
$\omega_A\in \Omega$ such that for any $(t',x',s'), (t'',x'',s'')\in
\mathbb{E}^\natural\cap [t_0,\vartheta_0]\times A\times \mathbb{R}^n$ the following inequality is fulfilled
 $$\|h^\natural(t',x',s')-h^\natural(t'',x'',s'')\|\leq \omega_A(t'-t'')+L_A\|x'-x''\|+\Gamma (1+\inf\{\|x'\|,\|x''\|\})\|s'-s''\|.$$
 \end{itemize}\end{list}

Condition (E4) is a restriction of conditions H1 and H2 on the set $\mathbb{E}$.

The proof of the main theorem is given in section 7. The proof uses lemmas formulated in sections 5 and 6. Let us introduce a method of extension of function $h$ from $\mathbb{E}_1$ to the set $\mathbb{E}$.

\begin{coll}\label{coll_suff}Let $\varphi\in \mathrm{Lip}_B$. Suppose that $h$
defined on $\mathbb{E}_1$ by formulas (\ref{h_def_nabla}) and (\ref{h_def_lim})
satisfies the condition {\rm (E1)}. Suppose also that the extension of $h$ on $\mathbb{E}_2$ given by the following rule is well defined:
 $\forall (t,x)\in CJ^-\cup CJ^+$, $s\in {E}_2(t,x)$
\begin{equation}\label{h_2_spec_case}
    h(t,x,s)\triangleq\sum_{i=1}^{n+2}\lambda_i h(t,x,s_i)
\end{equation} for any $s_1,\ldots,s_{n+2}\in {E}_1(t,x)$, $\lambda_1,\ldots,
\lambda_{n+2}$ such that $\sum \lambda_i=1$ $\sum\lambda_i s_i=s$. If function $h:\mathbb{E}\rightarrow\mathbb{R}$ satisfies the conditions {\rm (E3)} and {\rm (E4)}, then $\varphi\in \mathrm{VALF}$.
\end{coll}

The following corollaries is devoted to the relations between sets $\mathrm{VALF}$, $\mathrm{VALS}$ and $\mathrm{VALI}$. \begin{coll}\label{coll_eq}\label{Th_eq}Sets $\mathrm{VALF}$ and $\mathrm{VALS}$ coincide.
\end{coll}
\begin{coll}\label{Th_dim_1}If $n=1$, then $\mathrm{VALI}=\mathrm{VALF}=\mathrm{VALS}$.
\end{coll}
The corollaries are proved in section 7.

\section{Examples}
\textit{First (positive) example}.

Let $n=2$, $t_0=0$, $\vartheta_0=1$. Consider the function $$\varphi^1(t,x_1,x_2)\triangleq t+|x_1|-|x_2|.$$ Let us show that $\varphi^1(\cdot,\cdot,\cdot)\in\mathrm{VALF}$.

Function $\varphi^1$ is differentiated on the set $$J=\{(t,x_1,x_2)\in (0,1)\times \mathbb{R}^2:x_1,x_2\neq 0\}.$$ If $(t,x_1,x_2)\in J$, then $$\frac{\partial\varphi^1(t,x_1,x_2)}{\partial t}=1,\probel \nabla\varphi^1(t,x_1,x_2)=(\sgn x_1,-\sgn x_2). $$ Here $\sgn x$ means the sign of $x$:
$$\sgn x=\left\{\begin{array}{rc}
           1, & x> 0, \\
           -1, & x< 0.
         \end{array}\right.
 $$

Therefore, if $(\theta,g_1,g_2)\in\SDup{\varphi^1}{t,x_1,x_2}\cup \SDdn{\varphi^1}{t,x_1,x_2}$, then $\theta=1$.

Let us determine the set $E_1(t,x_1,x_2)\subset\mathbb{R}^n$ and function $h^1(t,x_1,x_2;s_1,s_2)$ for $(t,x_1,x_2)\in J$ and $(s_1,s_2)\in E_1(t,x_1,x_2)$. The representation of $J$ and formulas for partial derivatives of $\varphi^1$ yield the following representation of $E(t,x_1,x_2)$ and $h^1(t,x_1,x_2)$ for
${(t,x_1,x_2)\in J}$  $$E_1(t,x_1,x_2)=\{(\sgn x_1,-\sgn x_2)\},$$
$$h^1(t,x_1,x_2;\sgn x_1,-\sgn x_2)=-1. $$ Notice that condition (E1) for $\varphi^1$ is fulfilled.  Let ${(t,x_1,x_2)\notin J}$, then
$$E_1(t,x_1,x_2)=\left\{\begin{array}{rl}
                   \{(s_1,-\sgn x_2):|s_1|=1\}, & x_1=0,x_2\neq 0, \\
                   \{(\sgn x_1,s_2):|s_2|=1\}, & x_1\neq 0,x_2=0, \\
                   \{(s_1,s_2):|s_1|=|s_2|=1\}, & x_1=x_2=0.
                 \end{array}\right.
$$ If $(t,x_1,x_2)\notin J$, then for $(s_1,s_2)\in E_1(t,x_1,x_2)$ put $$h^1(t,x_1,x_2;s_1,s_2)=-1. $$ Now let us determine $\SDup{\varphi^1}{t,x_1,x_2}$ and $\SDdn{\varphi^1}{t,x_1,x_2}$ for $(t,x_1,x_2)\notin J$.

Let  $x_2\neq 0$, then $$\SDdn{\varphi^1}{t,0,x_2}=\{(1,s_1,-\sgn x_2):s_1\in [-1,1]\},\probel \SDup{\varphi^1}{t,0,x_2}=\varnothing. $$ Indeed, function $\varphi^1$ has directional derivatives at  points $(t,0,x_2)$ for $x_2\neq 0$. In addition, derivative in the direction $(\tau,g_1,g_2)$ is $$d\varphi^1(t,0,x_2;\tau,g_1,g_2)=\lim_{\alpha\rightarrow 0}\frac{\varphi^1(t+\alpha\tau,x_1+g_1\alpha,x_2+g_2\alpha)-\varphi^1(t,x_1,x_2)}{\alpha}=\tau+|g_1|-g_2\sgn x_2 . $$ We have,
\begin{multline*}
\{(1,s_1,-\sgn x_2):s_1\in [-1,1]\}=\{(\theta,s_1,s_2):(\theta\tau+s_1g_1+s_2g_2)\leq d\varphi^1(t,0,x_2;\tau,g_1,g_2)\}=\\=\SDdn{\varphi^1}{t,0,x_2}.
\end{multline*}

Similarly, for $x_1\neq 0$ we have $$\SDup{\varphi^1}{t,x_1,0}=\{(1,\sgn x_1,s_2):s_1\in [-1,1]\},\probel \SDdn{\varphi^1}{t,0,x_2}=\varnothing. $$

Further, $$\SDup{\varphi^1}{t,0,0}=\SDdn{\varphi^1}{t,0,0}=\varnothing. $$ Naturally, function $\varphi^1$ has directional derivatives at point $(t,0,0)$ and $$d\varphi^1(t,x_1,x_2;\tau,g_1,g_2)=\tau+|g_1|-|g_2|. $$ Suppose that $\SDup{\varphi^1}{t,0,0}\neq 0$. If $$(\theta,s_1,s_2)\in \SDup{\varphi^1}{t,x_1,x_2}, $$ then $$s_1 g_1\geq |g_1|\probel \forall g_1\in \mathbb{R}. $$ This yields that  $s_1\geq 1$ and $s_1\leq -1$. Thus, $\SDup{\varphi^1}{t,0,0}=\varnothing$. Similarly, $\SDdn{\varphi^1}{t,0,0}=\varnothing$.

Therefore, in this case $$ CJ^-=\{(t,0,x_2)\in (0,1)\times \mathbb{R}^2:x_2\neq 0\},$$ $$CJ^+=\{(t,x_1,0)\in (0,1)\times \mathbb{R}^2:x_1\neq 0\}. $$

We have $$E_2(t,x_1,x_2)=\left\{\begin{array}{lr}
                                      \{(1,s_1,-\sgn x_2):s_1\in [-1,1]\}, & x_1=0, x_2\neq 0,\\
                                      \{(1,\sgn x_1,s_2):s_1\in [-1,1]\}, & x_1\neq 0, x_2=0, \\
                                      \varnothing, & x_1x_2\neq 0,\mbox{ or }x_1=x_2=0.
                                    \end{array}\right.
$$ Use the corollary \ref{coll_suff}  to extend function  $h^1$ to the set $\mathbb{E}_2$.
Let $(t,x_1,x_2)\in CJ^-\cup CJ^+$, $s=(s_1,s_2)\in E_2(t,x_1,x_2)$, put $h^1(t,x_1,x_2,s_1,s_2)\triangleq -1.$ Since  for any $s'=(s_1',s_2')\in E_1(t,x_1,x_2)$ $h(t,x_1,x_2,s_1',s_2')=-1$, one can suppose that $h^1(t,x_1,x_2,s_1,s_2)$ is determined by (\ref{h_2_spec_case}).

Notice that condition (E3) is fulfilled since for any position $(t,x_1,x_2)$ the set $E(t,x_1,x_2)$ doesn't contain codirectional vectors  as well as vector $(0,0)$. It is easy to check that condition (E4) holds.

\textit{Second (negative) example.}

Let $n=2$, $t_0=0$, $\vartheta_0=1$. Let us show that $$\varphi^2(t,x_1,x_2)\triangleq t(|x_1|-|x_2|)\notin \mathrm{VALF}.$$

Function $\varphi^2(\cdot,\cdot,\cdot)$ is differentiated on the set $$J=\{(t,x_1,x_2)\in (0,1)\times \mathbb{R}^2:x_1x_2\neq 0\}. $$ We have $$\frac{\partial\varphi^2(t,x_1,x_2)}{\partial t}=|x_1|-|x_2|,\probel \nabla\varphi^2(t,x_1,x_2)=(t\cdot\sgn{x_1},-t\cdot\sgn{x_2}) $$ for $(t,x_1,x_2)\in J$.
Thus, for $(t,x_1,x_2)\in J$ $$h^2(t,x_1,x_2,t\cdot\sgn{x_1},t\cdot\sgn{x_2})=-(|x_1|-|x_2|),$$
$$E_1(t,x_1,x_2)=(t\cdot\sgn{x_1},-t\cdot\sgn{x_2}). $$ Further, if $(t,x_1,x_2)\in J$, $(s_1,s_2)\in E(t,x)$, then $\|(s_1,s_2)\|=t\sqrt{2}$. Thus for $(t,x_1,x_2)\in J$  the following equality is fulfilled $$E^\natural(t,x_1,x_2)=(\sgn{x_1}/\sqrt{2},-\sgn{x_2}/\sqrt{2}). $$
One can check directly that the condition (E1) holds in this case. Therefore we may suppose that $h^2(t,x_1,x_2,s_1,s_2)$ is defined on $\mathbb{E}_1$. Here we use formula (\ref{h_def_lim}).

Let us introduce the set $\mathbb{E}_0\subset(0,1)\times\mathbb{R}^2\times\mathbb{R}^2$.  Put
$$\mathbb{E}_0\triangleq\{(t,x_1,x_2,t\cdot\sgn{x_1},-t\cdot\sgn{x_2}):(t,x_1,x_2)\in J\}.$$ By definition of $\mathbb{E}$ we have $\mathbb{E}_0\subset\mathbb{E}$.

Suppose that there exists extension of the function $h^2$ satisfying the conditions (E2) and (E3). Hence the set $$\mathbb{E}_0^\natural\triangleq\{(t,x_1,x_2,\sgn{x_1}/\sqrt{2},-\sgn{x_2}/\sqrt{2}):(t,x_1,x_2)\in J\} $$ is subset of $\mathbb{E}^\natural$. Further, the function $(h^2)^\natural$ is well defined on $\mathbb{E}_0$. In this case $$(h^2)^\natural(t,x_1,x_2,\sgn{x_1}/\sqrt{x_1},-\sgn{x_2}/\sqrt{2})=\frac{|x_1|-|x_2|}{t\sqrt{2}}.$$ Obviously,  $(h^2)^\natural$ is unbounded on $(0,1)\times A\times \mathbb{R}^n\cap \mathbb{E}_0^\natural$. Here $A$ is any nonempty bounded subset of the set $\{(x_1,x_2)\in\mathbb{R}^n:x_1x_2\neq 0\}$). Hence, condition (E4) does not hold for any extension of  $h^2$. Thus $\varphi^2\notin\mathrm{VALF}$.

\section{Extension of $h$ to the whole space}
This section is devoted to the extension of $h$ to the space
$\ir\times\mathbb{R}^n$. This result is based on  McShane theorem about extension of range of function \cite{McShane}.
\begin{Lm}\label{lm_McShane} Under conditions {\rm (E1)--(E4)} function $h:\mathbb{E}\rightarrow\mathbb{R}$ can be extended  to $\ir\times \mathbb{R}^n$ such that the extension satisfies the conditions  {\rm H1--H3}.
\end{Lm}
\begin{proof}
The extension of $h$ is designed by two stages. First we extend function $h^\natural:\mathbb{E}^\natural\rightarrow\mathbb{R}$ to $\ir\times S^{(n-1)}$.  Finally we complete a definition by positive homogeneously.

Let us define the  function $h^*:\ir\times S^{(n-1)}\rightarrow\mathbb{R}$. The function is designed to be a extension of $h^\natural$. In order to define $h^*$ we design sequence of sets $\{G_r\}_{r=0}^\infty$, $G_r\subset \ir\times S^{(n-1)}$, and sequence of functions $\{h_r\}_{r=0}^\infty$, $h^r:G_r\rightarrow \mathbb{R}$, possessing following properties.
\begin{list}{(G\arabic{tmp4})}{\usecounter{tmp4}}
  \item\label{cond_G_0} $G_0=\mathbb{E}^\natural$, $h_0=h^\natural$
  \item\label{cond_inclusion} $G_{r-1}\subset G_r$ for all $r\in \mathbb{N}$.
  \item\label{cond_union} $\bigcup_{r=0}^\infty G_r=\ir\times S^{(n-1)}; $
  \item\label{cond_extension} for every natural number $r$ the restriction of $h_{r}$ on $G_{r-1}$ coincides with $h_{r-1}$;
  \item\label{cond_sub_lin} for any $(t,x,s)\in G_r$ the following inequalities is fulfilled:
$$|h_r(t,x,s)|\leq \Gamma(1+\|x\|),$$
\item\label{cond_continuity} for every $r\in \mathbb{N}_0$ and every bounded set $A\subset \mathbb{R}^n$
there exist constant $L_{A,r}$ and function $\omega_{A,r}\in \Omega$  such that
for any $(t',x',s'), (t'',x'',s'')\in G_r\cap [t_0,\vartheta_0]\times A\times S^{(n-1)}$ the following inequality is fulfilled:
\begin{multline}\label{h_lipshitz}
|h_r(t',x',s')-h_r(t'',x'',s'')|\leq \\ \leq
\omega_{A,r}(t'-t'')+L_{A,r}\|x'-x''\|+\Gamma(1+\inf\{\|x'\|,\|x''\|\})\|s'-s''\|.
\end{multline}
\end{list}
Here $\mathbb{N}_0\triangleq \mathbb{N}\cup\{0\}$.

We define function $h^*$ in the following way: for every
$(t,x,s)\in [t_0,\vartheta_0]\times\mathbb{R}^n\times S^{(n)}$ $h^*(t,x,s)=h_l(t,x,s)$.
Here $l$ is the least number $k\in\mathbb{N}_0$ such that $(t,x,s)\in G_k$.

Now let us define the sets $G_r$.  If $x\in \mathbb{R}^n$,
$j\in \overline{1,n}$, then by $x^j$ denote the $j$-th coordinate of $x$. By
$\|\cdot\|_*$ denote the following norm of $x$:
$$\|x\|_*\triangleq\max_{j=\overline{1,n}}|x^j|.$$
If $x\in\mathbb{R}^n$,
then \begin{equation}\label{zv_usual_relation}
    \|x\|_*\leq \|x\|.
\end{equation} Indeed, $$\|x\|=\sqrt{\sum_{j=1}^n(x^j)^2}\geq \sqrt{\max_j(x^j)^2}. $$
Let $e\in \mathbb{Z}^n$, let $a\in [0,\infty)$. ($\mathbb{Z}$ means the set of integer numbers.)
By $\Pi(e,a)$ denote $n$-dimensional cube with center at $e$ and length of edge which is equal to $a$:
$$\Pi(e,a)\triangleq\left\{x\in \mathbb{R}^n: \|e-x\|_*\leq \frac{a}{2}\right\}.
$$ If $a\geq 1$, then $$\mathbb{R}^n=\bigcup_{e\in\mathbb{Z}^n}\Pi(e,a). $$
Order elements $e\in \mathbb{Z}^n$, such that the following implication holds:
if  $\|e_i\|_*\leq \|e_k\|_*$, then $i\leq k$. Define the sequence $\{G_r\}_{r=0}^\infty$ by the rule:
\begin{equation}\label{G_def}
G_0\triangleq
\mathbb{E}^\natural,\probel G_k\triangleq G_{k-1}\cup ([t_0,\vartheta_0]\times\Pi(e_k,1)\times
S^{(n-1)})\probel\forall k\in\mathbb{N}.
\end{equation}
We have
$$[t_0,\vartheta_0]\times\mathbb{R}^n\times
S^{(n-1)}=\bigcup_{k\in\mathbb{N}_0}G_k.$$
Thus conditions (G\ref{cond_G_0})--(G\ref{cond_union}) are fulfilled by definition.

Now let us determine sequence of functions $\{h_r\}$. Put $$h_0(t,x,s)\triangleq h^\natural(t,x,s)\probel \forall (t,x,s)\in G_0=\mathbb{E}^\natural.$$
Notice that for $r=0$ conditions (G\ref{cond_sub_lin}) and (G\ref{cond_continuity}) are fulfilled by (E4).

Now suppose that function $h_{k-1}$ is determined on $G_{k-1}$ such that conditions
(G\ref{cond_sub_lin}) and (G\ref{cond_continuity}) hold with $r=k-1$. Let us determine function $h_k:G_k\rightarrow\mathbb{R}$.

Denote by $L_k$ the constant $L_{A,k-1}$ in the condition (G\ref{cond_continuity}) with $A=\Pi(e_k,3)$. We may assume that
\begin{equation}\label{Lambda_neq}
L_k\geq \Gamma.
\end{equation} By $\omega_k$ we denote the function $\omega_{A,k-1}$ with $A=\Pi(e_k,3)$.

Let $(t,x,s)\in G_k$. For $(t,x,s)\notin \pis{k}{1}$ put $h_k(t,x,s)\triangleq
h_{k-1}(t,x,s)$. For $(t,x,s)\in \pis{k}{1}$ put \begin{multline}\label{h_k_definition}h_k(t,x,s)\triangleq
\max\{-\Gamma(1+\|x\|),\\\sup\{h_{k-1}(\tau,y,\xi)-\omega_k(t-\tau)-L_k\|x-y\|-\Gamma(1+\|x\|)\|s-\xi\|:\\(\tau,y,\xi)\in
G_{k-1}\cap (\pis{k}{3})\}\}.\end{multline}

Let us show that the condition (G\ref{cond_extension}) is fulfilled for $r=k$. This means that  $h_k(t,x,s)=h_{k-1}(t,x,s)$ for  $(t,x,s)\in G_{k-1}\cap(\pis{k}{1})$. We have
\begin{multline*}
    \sup\{h_{k-1}(\tau,y,\xi)-\omega_k(t-\tau)-L_k\|x-y\|-\Gamma(1+\|x\|)\|s-\xi\|:\\(\tau,y,\xi)\in
G_{k-1}\cap (\pis{k}{3})\}\geq h_{k-1}(t,x,s)\geq-\Gamma(1+\|x\|).
\end{multline*}
Hence,
 \begin{multline}\label{h_k_k_minus_1}
   h_k(t,x,s)= \sup\{h_{k-1}(\tau,y,\xi)-\omega_k(t-\tau)-L_k\|x-y\|-\Gamma(1+\|x\|)\|s-\xi\|:\\(\tau,y,\xi)\in
G_{k-1}\cap (\pis{k}{3})\}\geq h_{k-1}(t,x,s).
\end{multline}
 Let
$\varepsilon>0$, let $(\tau,y,\xi)\in G_k\cap(\pis{k}{3})$ be an element satisfying the inequality
\begin{equation}\label{h_eps} h_k(t,x,s)\leq
h_{k-1}(\tau,y,\xi)-\omega_k(t-\tau)-L_k\|x-y\|-\Gamma(1+\|x\|)\|s-\xi\|+\varepsilon.
\end{equation}
Using (\ref{h_lipshitz}) with $r=k-1$ and $A=\Pi(e_k,3)$, we obtain
$$
h_{k-1}(\tau,y,\xi)-h_{k-1}(t,x,s)\leq
\omega_k(t-\tau)+L_k\|x-y\|+\Gamma(1+\inf\{\|x\|,\|y\|\})\|s-\xi\|.
$$ 
This and formula (\ref{h_eps}) yield the following estimate:
$$h_{k}(t,x,s)-h_{k-1}(t,x,s)\leq \varepsilon.$$ Since $\varepsilon$ is arbitrary
we obtain that $h_{k}(t,x,s)\leq h_{k-1}(t,x,s)$ for $(t,x,s)\in G_{k-1}\cap(\pis{k}{1})$.
The opposite inequality is established above (see (\ref{h_k_k_minus_1})). Therefore,
if {$(t,x,s)\in G_{k-1}\cap (\pis{k}{1})$}, then ${h_{k}(t,x,s)=h_{k-1}(t,x,s)}$. Thus function $h_k$ is an extension of $h_{k-1}$.

Moreover, one can prove the following implication: if  $(t,x,s)\in G_{k-1}\cap (\pis{k}{3})$, then
\begin{multline}\label{h_eq_for_3}h_{k-1}(t,x,s)
=\sup\{h_{k-1}(\tau,y,\xi)-\omega_k(t-\tau)-L_k\|x-y\|-\Gamma(1+\|x\|)\|s-\xi\|:\\(\tau,y,\xi)\in
G_{k-1}\cap (\pis{k}{3})\}.
\end{multline}

Let $(t,x,s)\in G_k\cap(\pis{k}{3})$. We shall say that the sequence $\{(t_i,x_i,s_i)\}_{i=1}^\infty\subset G_{k-1}\cap
(\pis{k}{3})$ realizes the value of $h_k(t,x,s)$, if
\begin{equation}\label{realized_seq}
    h_k(t,x,s)
    =\lim_{i\rightarrow\infty}[h_{k-1}(t_i,x_i,s_i)-\omega_k(t-t_i)-L_k\|x-x_i\|-\Gamma(1+\|x\|)\|s-s_i\|].
\end{equation}

If $h_k(t,x,s)> -\Gamma(1+\|x\|)$, then at least one sequence realizing the value of $h_k(t,x,s)$ exists (see (\ref{h_k_definition})).

Now we prove that $h_k$ satisfies the condition (G\ref{cond_sub_lin}) for $r=k$. Obviously, we may consider only triples $(t,x,s)\in \pis{k}{1}$.
 If $h_k(t,x,s)=-\Gamma(1+\|x\|)$, then the sublinear growth condition holds. Now let $h_k(t,x,s)>-\Gamma(1+\|x\|)$. Let sequence $\{(\tau_i,y_i,\xi_i)\}_{i=1}^\infty\subset
G_{k-1}\cap(\pis{k}{3})$ realize the value of $h_k(t,x,s)$.
Using inequality (\ref{Lambda_neq}) we obtain
\begin{multline*}
h_{k-1}(\tau_i,y_i,\xi_i)-\omega_k(t-\tau_i)-L_k\|x-y_i\|-\Gamma(1+\|x\|)\|s-\xi_i\|\leq \\
\leq \Gamma(1+\|y_i\|)-L_k\|x-y_i\|\leq \Gamma(1+\|x\|)+\Gamma\|x-y_i\|-L_k\|x-y_i\|\leq
\Gamma(1+\|x\|).
\end{multline*}
Consequently (see \ref{Lambda_neq}), the condition (G\ref{cond_sub_lin}) holds for $r=k$.

Let us show that $h_k$ satisfies the condition
(G\ref{cond_continuity}) for  $r=k$.
Let  $A$ be a bounded subset of $\mathbb{R}^n$, let $(t',x',s')$, $(t'',x'',s'')\in ([t_0,\vartheta_0]\times A\times S^{(n-1)})\cap G_k$. We estimate the difference  $h_k(t',x',s')-h_k(t'',x'',s'')$.

Let us consider 3 cases.
\begin{list}{\roman{tmp4}.}{\usecounter{tmp4}}
\item  $(t',x',s'),\ \ (t'',x'',s'')\ \ \notin \pis{k}{1}$.

Since $h_k(t,x,s)=h_{k-1}(t,x,s)$ for $(t,x,s)\in G_{k}\setminus \pis{k}{1}$, we have
    \begin{multline}\label{h_lip_oo}
    h_k(t',x',s')-h_k(t'',x'',s'')\leq \\ \leq \omega_{A,k-1}(t'-t'')+L_{A,k-1}\|x'-x''\|+\Gamma(1+\inf\{\|x'\|,\|x''\|\})\|s'-s''\|.
    \end{multline}

\item $(t',x',s'),(t'',x'',s'')\in \pis{k}{3}$ and at least one triple is in $\pis{k}{1}$.

From the definition of $h_k$ it follows that two subcases are possible.
    \begin{itemize}
        \item $h_k(t',x',s')=-\Gamma(1+\|x'\|)$.  In this case
        \begin{equation}\label{h_lip_iii_1}
h_k(t',x',s')-h_k(t'',x'',s'')\leq  -\Gamma(1+\|x'\|)+\Gamma(1+\|x''\|)\leq
\Gamma\|x''-x'\|.
\end{equation}
        \item $h_k(t',x',s')>-\Gamma(1+\|x'\|)$. 
        Let the sequence $\{(t_i,x_i,s_i)\}_{i=1}^\infty\subset G_{k-1}\cap(\pis{k}{3})$
        realize the value of $h_k(t',x',s')$. By (\ref{h_eq_for_3}) for $(t,x,s)=(t',x',s')$ and inequality $\|s''-s'\|,\|s'-s_i\|\leq 2$ we have
\begin{multline*}
    h_{k-1}(t_i,x_i,s_i)-\omega_k(t'-t_i)-L_k\|x'-x_i\|-\Gamma(1+\|x'\|)\|s'-s_i\|-h_k(t'',x'',s'')\leq\\
\leq
h_{k-1}(t_i,x_i,s_i)-\omega_k(t'-t_i)-L_k\|x'-x_i\|-\Gamma(1+\|x'\|)\|s'-s_i\|-\\-h_{k-1}(t_i,x_i,s_i)+
\omega_k(t''-t_i)+L_k\|x''-x_i\|+\Gamma(1+\|x''\|)\|s''-s_i\|
\leq \\
\leq \omega_k(t'-t'')+L_k\|x'-x''\|+\Gamma(1+\|x''\|)(\|s''-s_i\|-\|s'-s_i\|)+\\+\Gamma(\|x''\|-\|x'\|)\|s'-s_i\|\leq
\\
\leq\omega_k(t'-t'')+L_k\|x'-x''\|+\Gamma(1+\|x''\|)\|s'-s''\|+2\Gamma\|x'-x''\|\leq \\
\leq \omega_k(t'-t'')+(L_k+4\Gamma)\|x'-x''\|+\Gamma(1+\inf\{\|x'\|,\|x''\|\})\|s'-s''\|.
\end{multline*}
Hence, \begin{multline}\label{h_lip_iii_2}
    h_k(t',x',s')-h_k(t'',x'',s'')\leq \\
\omega_k(t'-t'')+(L_k+4\Gamma)\|x'-x''\|+\Gamma(1+\inf\{\|x'\|,\|x''\|\})\|s'-s''\|.
\end{multline}
    \end{itemize}
\item One of triples $(t',x',s')$, $(t'',x'',s'')$ belongs to $\pis{k}{1}$, and another triple doesn't belong to  $\pis{k}{3}$.

    Therefore, $\|x'-x''\|\geq\|x'-x''\|_*> 1$ (see (\ref{zv_usual_relation})). Since condition (G\ref{cond_sub_lin}) for $r=k$ is established above, we have
    \begin{equation}\label{h_lip_II}
    h(t',x',s')-h(t'',x'',s'')\leq 2\Gamma(1+\sup_{y\in A}\|y\|)\leq 2\Gamma(1+\sup_{y\in A}\|y\|)\|x'-x''\|.
    \end{equation}

\end{list}

The estimates (\ref{h_lip_oo})--(\ref{h_lip_II}) yield that if $(t',x',s'), (t'',x'',s'')\in G_k\cap ([t_0,\vartheta_0]\times A\times S^{(n-1)})$, then
\begin{equation}\label{h_res}
    h_k(t',x',s')-h_k(t'',x'',s'')\leq \omega_{A,k}(t'-t'')+L_{A,k}\|x'-x''\|+\Gamma(1+\inf\{\|x'\|,\|x''\|\})\|s'-s''\|.
\end{equation} Here $\omega_{A,k}$ is defined by the rule
$$\omega_{A,k}(\delta)\triangleq\max\{\omega_{A,k-1}(\delta),\omega_{k}(\delta)\} $$ (one can check directly that $\omega_{A,k}\in \Omega$); the constant $L_{A,k}$ is defined by the rule $$L_{A,k}\triangleq\max\left\{L_{A,k-1},L_k+4\Gamma,\Gamma(1+\sup_{y\in A}\|y\|)\right\}.$$

Therefore the condition (G\ref{cond_continuity}) is fulfilled for $r=k$.

This completes the designing of sequences $\{G_r\}_{r=0}^\infty$ and
$\{h_r\}_{r=0}^\infty$ satisfying the conditions (G\ref{cond_G_0})--(G\ref{cond_continuity}).

For every $(t,x,s)\in[t_0,\vartheta_0]\times \mathbb{R}^n\times S^{(n-1)}$ there exists number
$k\in\mathbb{N}_0$ such that $(t,x,s)\in G_k$. Put $$h^*(t,x,s)\triangleq
h_k(t,x,s).$$ The value of  $h^*(t,x,s)$ doesn't depend on number $k$ satisfying the property
$(t,x,s)\in G_k$. By definition of $h_k$ (see
(G\ref{cond_sub_lin})) we have
$$h^*(t,x,s)\leq \Gamma(1+\|x\|).
$$ Let us prove that for every bounded set $A\subset\mathbb{R}^n$ there exist function $\omega_A\in \Omega$ and constant
$L_A$ such that for all $(t',x',s'), (t'',x'',s'')\in [t_0,\vartheta_0]\times A\times
S^{(n-1)}$ the following estimate is fulfilled \begin{equation}\label{h_star_lip} |h^*(t',x',s')-h^*(t'',x'',s'')|\leq
\omega_A(t'-t'')+L_A\|x'-x''\|+\Gamma(1+\inf\{\|x'\|,\|x''\|\})\|s'-s''\|.
\end{equation}
Indeed, there exists number $m$ such that $$A\subset \bigcup_{k=1}^m\Pi(e_k,1).$$ By definition of
$\{G_k\}$ (see (\ref{G_def})) we have
$$[t_0,\vartheta_0]\times A\times S^{(n-1)}\probel\subset\probel [t_0,\vartheta_0]\times\left[\bigcup_{k=1}^m\Pi(e_k,1)\right]\times
S^{(n-1)}\subset G_m.
$$ Put $\omega_A\triangleq\omega_{A,m}$, $L_A\triangleq L_{A,m}$. Since
$h^*(t,x,s)=h_m(t,x,s)$ $\forall (t,x,s)\in [t_0,\vartheta_0]\times A\times S^{(n-1)}$, the property (G\ref{cond_continuity}) for $r=m$ yields that \begin{multline*}
    |h^*(t',x',s')-h^*(t'',x'',s'')|=|h_m(t',x',s')-h_m(t'',x'',s'')|\leq\\\leq
\omega_{A,m}(t'-t'')+L_{A,m}\|x'-x''\|+\Gamma(1+\inf\{\|x'\|,\|x''\|\})\|s'-s''\|.
\end{multline*}
Thus, the inequality (\ref{h_star_lip}) is fulfilled.

Now let us introduce the function $H:[t_0,\vartheta_0]\times \mathbb{R}^n\times \mathbb{R}^n\rightarrow\mathbb{R}$.
Put \begin{equation}\label{H_large_def}H(t,x,s)\triangleq\left\{
\begin{array}{cr}
  \|s\|h^*(t,x,\|s\|^{-1}s), & s\neq 0 \\
  0, & s=0. \\
\end{array}\right.
\end{equation}

Function $H$ is an extension of $h$. Naturally, let $(t,x,s)\in \mathbb{E}$, $s\neq 0$. Then
$(t,x,\|s\|^{-1}s)\in \mathbb{E}^\natural$. Hence,
$$ H(t,x,s)=\|s\|h^*(t,x,\|s\|^{-1}s)=\|s\|h^\natural(t,x,\|s\|^{-1}\|s\|)=h(t,x,s).
$$ If $(t,x,0)\in\mathbb{E}$, then by condition (E3) we have $$h(t,x,0)=0=H(t,x,0). $$

Function $H$ satisfies the condition H2. Let $s_1,s_2\in \mathbb{R}^n$, $(t,x)\in\ir$. Let us estimate  $|H(t,x,s_1)-H(t,x,s_2)|$.  Without loss of generality it can be assumed that   $\|s_1\|\geq \|s_2\|$. If $\|s_2\|=0$, then
 \begin{equation}\label{H_0_estima} |H(t,x,s_1)-H(t,x,s_2)|=|H(t,x,s_1)|\leq
\Gamma(1+\|x\|)\|s_1\|=\Gamma(1+\|x\|)\|s_1-s_2\|.
\end{equation} Now let $\|s_2\|>0$.
\begin{multline}\label{H_1_estima}
    |H(t,x,s_1)-H(t,x,s_2)|=\left|\|s_1\|h^*\left(t,x,\frac{s_1}{\|s_1\|}\right)-
    \|s_2\|h^*\left(t,x,\frac{s_2}{\|s_2\|}\right)\right|\leq\\\leq
(\|s_1-s_2\|)\left|h^*\left(t,x,\frac{s_1}{\|s_1\|}\right)\right|+\|s_2\|\left|h^*\left(t,x,\frac{s_1}{\|s_1\|}\right)-h^*\left(t,x,\frac{s_2}{\|s_2\|}\right)\right|\leq\\
\leq
\Gamma(1+\|x\|)\|s_1-s_2\|+\|s_2\|\Gamma(1+\|x\|)\left\|\frac{s_1}{\|s_1\|}-\frac{s_2}{\|s_2\|}\right\|\leq\\
\leq 2\Gamma(1+\|x\|)\|s_1-s_2\|.
\end{multline} In order to prove the last estimate in (\ref{H_1_estima}) we need to show that if $\|s_1\|\geq
\|s_2\|$ then \begin{equation}\label{s_1_s_2}
\left\|\frac{\|s_2\|s_1}{\|s_1\|}-s_2\right\|\leq \|s_1-s_2\|.
\end{equation}
Let $z\in \mathbb{R}^n$ be a codirectional with $s_1$, let $\gamma$ be the angle between
$s_1$ and $s_2$: $$\cos\gamma=\frac{\langle s_1,s_2\rangle}{\|s_1\|\cdot\|s_2\|}. $$
Consider triangle formed by the origin and terminuses of $z$ and $s_2$.
The lengths of side of triangle are $\|z\|$, $\|s_2\|$ and $\|z-s_2\|$. By the cosine theorem we have
$$\|z-s_2\|^2=\|s_2\|^2+\|z\|^2-2\|z\|\|s_2\|\cos\gamma=\|s_2\|^2(1-\cos^2\gamma)+(\|z\|-\|s_2\|\cos\gamma)^2.
$$ Hence, the function  $\|z-s_2\|$ as a function of $\|z\|$ increases on the region $\|z\|\geq
\|s_2\|\cos\gamma$. Since $$\left\|\frac{\|s_2\|s_1}{\|s_1\|}\right\|=\|s_2\|\leq
\|s_1\|,
$$ the estimate (\ref{s_1_s_2}) holds.

Combining estimates (\ref{H_0_estima}) and (\ref{H_1_estima}) we get
\begin{equation}\label{H_final_estima}
    |H(t,x,s_1)-H(t,x,s_2)|\leq \Upsilon(1+\|x\|)\|s_1-s_2\| \probel
    \forall(t,x)\in\ir\probel\forall s_1,s_2\in\mathbb{R}^n.
\end{equation} Here $\Upsilon=2\Gamma$.
Using the definition of $H$ (see (\ref{H_large_def})), properties of function $h^*$ (see (\ref{h_star_lip})), we obtain that  function $H$ satisfies the condition H2.

Notice that for all $(t,x,s)\in [t_0,\vartheta_0]\times\mathbb{R}^n\times\mathbb{R}^n$ the following inequality holds:
$$|H(t,x,s)|\leq \Gamma\|s\|(1+\|x\|)\leq\Upsilon \|s\|(1+\|x\|). $$ This means that the function $H$ satisfies the condition H1.

Function $H$ is positively homogeneous by definition.

This completes the proof.
\end{proof}

\section{Construction of Game with the Given Hamiltonian}
The following lemma is close to the result of L.C.Evans and
P.E.Souganidis (see \cite{Evans}) about construction of differential games. We consider unbounded, locally lipschitzian hamiltonians but in  \cite{Evans} only bounded on $\ir\times S^{(n-1)}$, uniformly lipschitzian hamiltonians are considered.
\begin{Lm}\label{lm_Evans}
Let function $H:[t_0,\vartheta_0]\times\mathbb{R}^n\times\mathbb{R}^n\rightarrow
\mathbb{R}$ satisfy the conditions {\rm H1--H3}. Then there exist sets
$P,Q\in\mathrm{COMP}$ and function $f\in \mathrm{DYN}(P,Q)$ such that
\begin{equation}\label{H_Evans_prop} H(t,x,s)=\max_{v\in Q}\min_{u\in P}\langle
s,f(t,x,u,v)\rangle\probel \forall (t,x,s)\in\ir\times \mathbb{R}^n.
\end{equation}
\end{Lm}
\begin{proof}
Denote $$B\triangleq\{s\in\mathbb{R}^n:\|s\|\leq 1\}.$$ By the condition H2 there exists a real number $\Upsilon$, such that for all $(t,x,s_1), (t,x,s_2)\in\ir\times \mathbb{R}^n$ the following estimate holds: $$|H(t,x,s_1)-H(t,x,s_2)|\leq \Upsilon(1+\|x\|)\|s_1-s_2\|. $$

Therefore,
\begin{multline*}
    H(t,x,s)=\|s\|H\left(t,x,\frac{s}{\|s\|}\right)=\|s\|\max_{z\in
    B}\left[H(t,x,z)-\Upsilon(1+\|x\|)\left\|\frac{s}{\|s\|}-z\right\|\right]=\\=
\|s\|\max_{z\in B}\min_{y\in B}\left[H(t,x,z)+\Upsilon(1+\|x\|)\left\langle
y,\frac{s}{\|s\|}-z\right\rangle\right]=\\= \|s\|\max_{z\in B}\min_{y\in
B}\left[(H(t,x,z)+\Upsilon(1+\|x\|))-\Upsilon(1+\|x\|)+\Upsilon(1+\|x\|)\left\langle
y,\frac{s}{\|s\|}-z\right\rangle\right]=\\= \max_{z\in B}\min_{y\in
B}[(H(t,x,z)+\Upsilon(1+\|x\|)\|s\|+\Upsilon(1+\|x\|)\langle
y,s\rangle-\Upsilon(1+\|x\|)(1+\langle y,z\rangle)\|s\|]
\end{multline*}

Since for all $y,z\in B$  $$H(t,x,z)+\Upsilon(1+\|x\|),\probel \Upsilon(1+\|x\|)(1+\langle y,z\rangle)\geq 0,
$$ it follows that
\begin{multline}\label{H_min_max_first}
    H(t,x,s)=\max_{z\in B}\min_{y\in
B}\max_{z'\in B}\min_{y'\in B}\\
[(H(t,x,z)+\Upsilon(1+\|x\|))\langle z',s\rangle+\Upsilon(1+\|x\|)\langle
y,s\rangle+\Upsilon(1+\|x\|)(1+\langle y,z\rangle)\langle y',s\rangle].
\end{multline} In formula (\ref{H_min_max_first}) one can 	interchange $\min_{y\in B}$ and $\max_{z'\in
B}$. Denoting $P=Q=B\times B$, and  
$$f(t,x,u,v)\triangleq
H(t,x,z)z'+\Upsilon(1+\|x\|)[z'+y+(1+\langle y,z\rangle) y'],$$ for $(t,x)\in \ir$, $u=(y,y')$,
$v=(z,z')$
we obtain that (\ref{H_Evans_prop}) is fulfilled. By definition of $f$ it follows that $f\in \mathrm{DYN}(P,Q)$.

\end{proof}

\begin{Lm}\label{lm_evans_min_max}
Let function
$H:[t_0,\vartheta_0]\times\mathbb{R}^n\times\mathbb{R}^n\rightarrow \mathbb{R}$
satisfy the conditions {\rm H1--H3}. Then there exist sets $P,Q\in\mathrm{COMP}$ and a function $f\in \mathrm{DYN}(P,Q)$ such that $$ H(t,x,s)=\min_{u\in P}\max_{v\in Q}\langle
s,f(t,x,u,v)\rangle\probel \forall (t,x,s)\in\ir\times \mathbb{R}^n.
$$
\end{Lm}
\begin{proof} $$H(t,x,s)=\|s\|\min_{y\in B}\left[H(t,x,y)+\Upsilon(1+\|x\|)\left\|\frac{s}{\|s\|}-y\right\|\right]. $$
Then the proof is similar  to the proof of previous lemma.
\end{proof}

\begin{Lm}\label{lm_evans_1} Let $n=1$, $H:[t_0,\vartheta_0]\times\mathbb{R}^n\times\mathbb{R}^n\rightarrow \mathbb{R}$
satisfy the conditions {\rm H1--H3}. Then there exist sets $P,Q\in\mathrm{COMP}$ and a function
$f\in \mathrm{DYNI}(P,Q)$ such that \begin{multline}\label{H_Evans_prop_min_max}
H(t,x,s)=\max_{v\in Q}\min_{u\in P}\langle s,f(t,x,u,v)\rangle=\min_{u\in P}\max_{v\in
Q}\langle s,f(t,x,u,v)\rangle\\ \forall (t,x,s)\in\ir\times \mathbb{R}^n.
\end{multline}
\end{Lm}
\begin{proof}
If $s\neq 0$, then
\begin{multline}\label{H_n_1_repres}
    H(t,x,s)=\max_{z\in\{-1,1\}}[H(t,x,z)\|s\|-\Upsilon(1+\|x\|)\langle s, \|s\|^{-1}s-z\rangle]=\\=
\max_{z\in\{-1,1\}}\min_{y\in \{-1,1\}}[H(t,x,z)\|s\|+\Upsilon(1+\|x\|)\|s\|\langle y, \|s\|^{-1}s-z\rangle]=\\=
\max_{z\in \{-1,1\}}\min_{y\in \{-1,1\}}[(H(t,x,z)+\Upsilon(1+\|x\|))\|s\|+\Upsilon(1+\|x\|)(\langle y, s\rangle-\|s\|(\langle y,z\rangle+1))]=\\=
\max_{z\in \{-1,1\}}\min_{y\in \{-1,1\}}\max_{z'\in\{-1,1\}}\min_{y'\in \{-1,1\}}
\\
[(H(t,x,z)+\Upsilon(1+\|x\|))\langle z',s\rangle+\Upsilon(1+\|x\|)\langle y, s\rangle+\Upsilon(1+\|x\|)(\langle y,z\rangle+1)\langle y',s\rangle]
\end{multline}

Note that $\min_{y\in \{-1,1\}}$ and $\max_{z'\in \{-1,1\}}$ are permutable. Denote $$g(t,x,y,y',z,z')\triangleq (H(t,x,z)+\Upsilon(1+\|x\|)) z'+\Upsilon(1+\|x\|)y+\Upsilon(1+\|x\|)(\langle y,z\rangle+1) y'.$$ Thus, for $s\in \mathbb{R}$ the following representation is fulfilled:
$$H(t,x,s)=\max_{z,z'\in \{-1,1\}}\min_{y,y'\in \{-1,1\}}\langle s,g(t,x,y,y',z,z')\rangle. $$

In addition, by (\ref{H_n_1_repres}) and definition of $g$ we obtain for $s\neq 0$ the following representation:
\begin{multline*}
    H(t,x,s)=\max_{z,z'\in \{-1,1\}}\min_{y,y'\in \{-1,1\}}\langle s,g(t,x,y,y',z,z')\rangle=\\=
\max_{z,z'\in \{-1,1\}}\langle s,g(t,x,-\|s\|^{-1}s,-\|s\|^{-1}s,z,z')\rangle\geq
\min_{y,y'\in \{-1,1\}}\max_{z,z'\in \{-1,1\}}\langle s,g(t,x,y,y',z,z')\rangle.
\end{multline*} Therefore, for all $s\in \mathbb{R}$ the following inequality holds
$$
 H(t,x,s)=\max_{z,z'\in \{-1,1\}}\min_{y,y'\in \{-1,1\}}\langle s,g(t,x,y,y',z,z') \rangle \geq
\min_{y,y'\in \{-1,1\}}\max_{z,z'\in \{-1,1\}}\langle s,g(t,x,y,y',z,z')\rangle.
$$

Opposite inequality is obvious.

Denote $P=Q=\{-1,1\}\times \{-1,1\}$, $u=(y,y')$, $v=(z,z')$,
$f(t,x,u,v)=g(t,x,y,y',z,z')$. We have, $f\in \mathrm{DINI}$ and $$H(t,x,s)=\max_{v\in Q}\min_{u\in P}\langle s, f(t,x,u,v)\rangle. $$

\end{proof}

\section{Construction of the Differential Games whose Value Coincides with a Given Function}
In this section we prove the statements formulated in the section~\ref{sect_main}.
\begin{proof}[Proof of the Main Theorem]
{\it Necessity}.

Let $\varphi\in\mathrm{Lip}_B\cap\mathrm{VALF}$. Then by definition of
$\mathrm{VALF}$ there exist the sets $P,Q\in \mathrm{COMP}$, and the function $f\in\mathrm{DYN}(P,Q)$,
$\sigma\in\mathrm{TP}$ such that
$\varphi=Val^f(\cdot,\cdot,P,Q,f,\sigma)$. Therefore (see \cite{Subb_book})
$\varphi$ is a minimax solution of the equation
$$ \frac{\partial\varphi}{\partial t}+H(t,x,\nabla\varphi)=0
$$ with $$H(t,x,s)=\max_{v\in Q}\min_{u\in P}\langle s,f(t,x,u,v)\rangle. $$ Consider the function $h$ defined by formula (\ref{h_def_nabla}) on $J$. Note that $J$ means the set of differentiability of
$\varphi$. We have $$h(t,x,s)=H(t,x,s),\probel (t,x)\in J, s\in E(t,x). $$ Let $(t,x)$ be a position at which  function $\varphi$ is 	nondifferentiable, $s\in E_1(t,x)$.
Denote
\begin{multline*}
L\varphi(t,x,s)\triangleq\{a\in\mathbb{R}:\exists \{(t_i,x_i)\}_{i=1}^\infty\subset J:\\\probel (t,x,s)=\lim_{i\rightarrow\infty}(t_i,x_i,\nabla\varphi(t_i,x_i))\probel\&\probel a=\lim_{i\rightarrow \infty}\partial\varphi(t_i,x_i)/\partial t\}.
\end{multline*}
Since  $\partial\varphi(t,x)/\partial t=-H(t,x,\nabla\varphi(t,x))$ for $(t,x)\in J$,
the continuity $H$ yields that $$L\varphi(t,x,s)=\{-H(t,x,s)\},\ \
(t,x)\notin J,\ \ s\in E_1(t,x).$$ Thus function $h=H$ satisfies the condition
(E1). In addition, function $h(t,x,s)=H(t,x,s)$ is determined by (\ref{h_def_lim}) for $(t,x)\notin J$, $s\in E_1(t,x)$. We have that
$h=H$ on $\mathbb{E}_1$.

Set the extension of $h$ to $\mathbb{E}_2$ to be equal to $H$.
Since $\varphi$ is minimax solution of Hamilton-Jacobi equation, we get that for all
$(t,x)\in \ir$ the following inequalities hold
$$a+H(t,x,s)\leq 0\probel \forall(a,s)\in D^-\varphi(t,x).
$$
$$a+H(t,x,s)\geq 0\probel \forall(a,s)\in D^+\varphi(t,x).
$$
If function $\varphi$ is not differentiable at  $(t,x)$ and $D^-\varphi(t,x)\cup D^+\varphi(t,x)\neq \varnothing$, then either  $(t,x)\in CJ^-$ or $(t,x)\in CJ^+$. Let $(t,x)\in CJ^-$. Consider
$\lambda_1,\ldots\lambda_{n+2}\in [0,1]$ and $s_1,\ldots,s_{n+2}\in E_1(t,x)$ such that
$\sum\lambda_i=1$ and
$$\left(-\sum_{k=1}^{n+2}\lambda_k H(t,x,s_k), \sum_{k=1}^{n+2}\lambda_k s_k\right)\in D^-(t,x). $$
Therefore,
$$-\sum_{k=1}^{n+2}\lambda_k H(t,x,s_k)+H\left(t,x,\sum_{k=1}^{n+2}\lambda_k s_k\right)\leq 0. $$
Similarly, if $(t,x)\in CJ^+$,
$\lambda_1,\ldots\lambda_{n+2}\in [0,1]$ and $s_1,\ldots,s_{n+2}\in E_1(t,x)$ satisfy the conditions
$\sum\lambda_i=1$
$$\left(-\sum_{k=1}^{n+2}\lambda_k H(t,x,s_k), \sum_{k=1}^{n+2}\lambda_k s_k\right)\in D^+(t,x),
$$ then the following inequality is fulfilled: $$-\sum_{k=1}^{n+2}\lambda_k H(t,x,s_k)+H\left(t,x,\sum_{k=1}^{n+2}\lambda_k s_k\right)\geq 0. $$

We get that function $h=H$ satisfies the condition (E2).

The condition (E3) holds since $H$ is positively homogeneous.
Note that $h^\natural(t,x,s)=H(t,x,s)\probel\forall (t,x,s)\in\mathbb{E}^\natural$.
Since $H$ satisfies the conditions H1 and H2, condition (E4) is fulfilled also.


\end{proof}
\begin{proof}[Proof of the Main Theorem]
{\it Sufficiency}.

Consider the function  $h$ is defined on $\mathbb{E}_1$ by formulas
(\ref{h_def_nabla}) and (\ref{h_def_lim}). By the assumption there exists the extension of $h$  to $\mathbb{E}$ which satisfies the conditions (E2)--(E4). By lemma \ref{lm_McShane}
there exists the function $H:\ir\times \mathbb{R}^n\rightarrow\mathbb{R}$ which is extension of $h$ and satisfies the conditions H1--H3. By lemma
\ref{lm_Evans} there exist compacts $P,Q\in\mathrm{COMP}$ and function
$f\in\mathrm{DYN}(P,Q)$ such that \begin{equation}\label{h_star_repres}
H(t,x,s)=\max_{u\in P}\min_{v\in Q}\langle s, f(t,x,u,v)\rangle.
\end{equation} Put
\begin{equation}\label{sigma_def} \sigma(x)\triangleq\varphi(\vartheta_0,x).
\end{equation} Since $\varphi\in\mathrm{Lip}_B$, we get $\sigma\in \mathrm{TP}$. Let us show that $\varphi=Val^f(\cdot,\cdot,P,Q,f,\sigma)$. This is equivalent to the requirement that $\varphi$ satisfies the conditions (\ref{boundar_cond}), (\ref{U4}) and (\ref{L4}).

Obviously, the boundary condition (\ref{boundar_cond}) is valid by definition of $\sigma$. Let us show that $\varphi$ the conditions (\ref{U4}) and (\ref{L4}) are valid.

If $(t,x)\in J$, then
$\SDdn{\varphi}{t,x}=\SDup{\varphi}{t,x}=\{(\partial\varphi(t,x)/\partial
t,\nabla\varphi(t,x))\}$
$$\frac{\partial\varphi(t,x)}{\partial t}=-h(t,x,\nabla\varphi(t,x))=-H(t,x,\nabla\varphi(t,x)).
$$ Therefore, for $(t,x)\in J$ the inequalities (\ref{U4}) and (\ref{L4}) hold.

Now consider $(t,x)\notin J$.
By the properties Clarke subdifferential and function $h$  (see (\ref{cl_repr_s1}), (\ref{incl_dini_clark})) it follows that
\begin{equation}\label{dini_clark_inkl_tx}
\SDdn{\varphi}{t,x},\SDup{\varphi}{t,x}\subset{\rm co}\{(-h(t,x,s),s):s\in E_1(t,x)\}.
\end{equation}

If $(a,s)\in \SDdn{\varphi}{t,x}$ (in this case $(t,x)\in CJ^-$), then there exist
$\lambda_1,\ldots,\lambda_{n+2}\in [0,1]$, $s_1,\ldots,s_{n+2}\in E_1(t,x)$ such that
$\sum\lambda_k=1$, $\sum\lambda_k s_k=s$, $-\sum \lambda_k h(t,x,s_k)=a$ (see  (\ref{dini_clark_inkl_tx})). Using condition (E2) we obtain  $$h(t,x,s)\leq \sum\lambda_k h(t,x,s_k)=-a. $$ This is equivalent to the condition (\ref{U4}).  Similarly the truth of (\ref{L4}) can be proved. Thus,
 $\varphi$ is minimax solution (\ref{HJ}) with boundary condition (\ref{boundar_cond}). By \cite{Subb_book} and (\ref{h_star_repres}) it follows that
$\varphi=Val^f(\cdot,\cdot,P,Q,f,\sigma)$. This completes the proof.

\end{proof}

\begin{proof}[Proof of Corollary \ref{coll_suff}] The condition (E1) is valid by the assumption.
If $(t,x)\in CJ^-$, $s\in E_2(t,x)$, then for any $\lambda_1,\ldots
\lambda_{n+2}\in [0,1]$, $s_1\ldots,s_{n+2}$ such that $\sum\lambda_k=1$, $\sum\lambda_k
s_k=s$, the following inclusion holds: $$\left(-\sum_{k=1}^{n+2}\lambda_kh(t,x,s_k),\sum_{k=1}^{n+2}\lambda_k
s_k\right)\in D^-\varphi(t,x).
$$ By assumption $$h(t,x,s)=\sum_{k=1}^{n+2}\lambda_kh(t,x,s_k). $$ Therefore the first part of condition (E2) is fulfilled.
In the same way the second part of (E2) can be proved. The conditions (E3) and (E4) hold by assumption. Therefore $\varphi\in\mathrm{VALF}$.

\end{proof}

\begin{proof}[Proof of Corollary \ref{coll_eq}]
Let $\varphi\in \mathrm{VALF}$. There exist sets $P,Q$ and function $f\in \mathrm{DYN}(P,Q)$, $\sigma\in \mathrm{TP}$ such that
$\varphi=Val^f(\cdot,\cdot,P,Q,f,\sigma)$. By lemma
\ref{lm_evans_min_max} there exist sets $P',Q'\in\mathrm{COMP}$ and function
$f'\in\mathrm{DYN}(P,Q)$ such that for any $(t,x)\in \ir$ $s\in\mathrm{R}^n$ the following equality holds: $$\max_{v\in Q}\min_{u\in P}\langle s,f(t,x,u,v)\rangle=H(t,x,s)=\min_{u\in
P'}\max_{v\in Q'}\langle s,f'(t,x,u,v)\rangle.
$$ Consequently, $\varphi=Val^s(\cdot,\cdot,P',Q',f',\sigma)\in
\mathrm{VALS}$. Thus,
$$ \mathrm{VALF}\subset\mathrm{VALS}.$$ The opposite inclusion is proved in the similar way.
\end{proof}

\begin{proof}[Proof of Corollary \ref{Th_dim_1}]
Obviously, $$\mathrm{VALI}\subset \mathrm{VALF}=\mathrm{VALS}.$$ We shall prove that if  $n=1$ then \begin{equation}\label{Isaacs_incl}
\mathrm{VALF}\subset \mathrm{VALI}.
\end{equation} Let $\varphi\in \mathrm{VALF}$. By definition of
$\mathrm{VALF}$ there exist sets $P,Q\in\mathrm{COMP}$ and functions
$f\in\mathrm{DYN}(P,Q)$, $\sigma\in \mathrm{TP}$ such that
$$\varphi=Val^f(\cdot,\cdot,P,Q,f,\sigma). $$ By lemma \ref{lm_evans_1}
there exist sets $P_1,Q_1\in\mathrm{COMP}$ and function
$f_1\in\mathrm{DYNI}(P,Q)$ such that
$$\min_{u\in P_1}\max_{v\in Q_1}\langle s,f_1(t,x,u,v)=H(t,x,s)=\max_{v\in Q}\min_{u\in P}\langle s,f(t,x,u,v)\rangle. $$
 Thus $\varphi= {Val}(\cdot,\cdot,P_1,Q_1,f_1,\sigma)$. Therefore, the inclusion
 (\ref{Isaacs_incl}) holds.
\end{proof}


\begin{thebibliography}{99}
\bibitem{NN_PDG_en} {\it Krasovskii N.N., Subbotin A.I.} Game-Theoretical Control
Problems, New York: Springer, 1988, 517 p.
\bibitem{Subb_book} {\it Subbotin A.I.} Generalized solutions of first-order PDEs. The dynamical perspective, Systems \& Control: Foundations \& Applications, Birkhauser, Boston, Ins., Boston MA, 1995, 312 p.

\bibitem{Bardi}{\it Bardi M, Capuzzo-Dolcetta I.} {Optimal control and
viscosity solutions of Hamilton-Jacobi-Bellman equations.
With appendices by Maurizio
Falcone and Pierpaolo Soravia, Boston.
Systems \& Control: Foundations \& Applications.
Birkhauser Boston, Inc. 1997, xviii+570 pp.}
\bibitem{DemRub} {\it Demyanov V.F., Rubinov A.M.} Foundations of Nonsmooth Analysis, and Quasidifferential Calculus, Optimization and Operation Research, v. 23, Nauka, Moscow, 1990, 431pp.
\bibitem{Evans_measure}  \textit{Evans L.C.,  Gariepy R.F.}, Measure Theory and Fine Properties of Functions: CRC, Boca Raton, 1992

\bibitem{McShane} {\it McShane E. J.} Extension of range of function // Bull.Amer.Math.Soc.
1934. V. 40. №12, Pp 837--842.
\bibitem{Evans} {\it Evans L.C., Souganidis P.E.} Differential games and representation
formulas for solutions of Hamilton-Jacobi-Isaacs Equations // Indiana University
Mathematical Journal, 1984, Vol. 33, N 5, Pp. 773--797.
\end{thebibliography}
\end{document}